\documentclass[12pt]{article}

\usepackage{amsthm}
\usepackage{amsmath}
\usepackage{amsfonts}
\usepackage{enumitem}
\usepackage{cite}
\usepackage{hyperref}
\usepackage{authblk}

\newtheorem{theorem}{Theorem}
\newtheorem{lemma}{Lemma}
\newtheorem{proposition}{Proposition}
\newtheorem{corollary}{Corollary}

\theoremstyle{definition}
\newtheorem{definition}{Definition}
\newtheorem{example}{Example}

\newcommand{\Z}[1]{\mathbb{Z}_{#1}}
\newcommand{\z}{\mathbb{Z}}

\begin{document}

\title{Results on Vanishing Polynomials and Polynomial Root Counting}
\author[1]{Matvey Borodin\thanks{Corresponding author: \href{mailto:matveyborodin1@gmail.com}{matveyborodin1@gmail.com}}}
\affil[1]{Brookline High School}
\author[2]{Ethan Liu}
\affil[2]{The Harker School}
\author[3]{Justin Zhang}
\affil[3]{Bergen County Academics}
\date{}
\maketitle

\begin{abstract}
We study the set of algebraic objects known as vanishing polynomials (the set of polynomials that annihilate all elements of a ring) over general commutative rings with identity. These objects are of special interest due to their close connections to both ring theory and the technical applications of polynomials, along with numerous applications to other mathematical and engineering fields. We first determine the minimum degree of monic vanishing polynomials over a specific infinite family of rings of a specific form and consider a generalization of the notion of a monic vanishing polynomial over a subring. We then present a partial classification of the ideal of vanishing polynomials over general commutative rings with identity of prime and prime square orders. Finally, we prove some results on rings that have a finite number of roots and propose a technique that can be utilized to restrict the number of roots polynomials can have over certain finite commutative rings.
\end{abstract}

\section{Introduction}

Polynomial rings are some of the most fundamental and extensively studied objects in mathematics. The study of vanishing polynomials, which are polynomials that evaluate to $0$ at every element of a ring, connects these polynomial rings to applications in technology. In particular, vanishing polynomials, as well as the closely related concept of polynomial functions, are used in vanishing component analysis, homomorphic encryption, and electrical engineering, as expanded on in \cite{livni2013vanishing}, \cite{geelen2023polynomial}, and \cite{meredith2007polynomial}, respectively.

% While there has been a significant amount of research done regarding polynomials over fields and integral domains, less has been done investigating polynomials over rings with zero divisors. In particular, it is widely known that a polynomial of degree $n$ can have at most $n$ roots over a field, but when zero-divisors are introduced this result may fail. More specifically, a polynomial that vanishes for all $x \in R$ must be of degree at least $|R|$ if $R$ is an integral domain, but if $R$ has zero divisors the degree of this polynomial can be significantly smaller. 

A complete description of such vanishing polynomials over the integers was given in \cite{SINGMASTER, Niven} while a more general result for multiple variables was found in \cite{GREUEL}. It is well known the set of vanishing polynomials is an ideal and the quotient of the original polynomial ring by this ideal gives the ring of polynomial functions. Such a ring of polynomial functions over $\Z{n}$, the ring of integer residue classes modulo $n$, was explored in \cite{specker, GUHA}. Throughout this paper, we build on these works and examine the structure of the ideal of vanishing polynomials and polynomial functions for more general rings.

In Section~\ref{vanishing_polynomials}, we begin with definitions of essential terms followed by background results that are used throughout the paper.
In Section~\ref{general_of_monic}, we develop the notion of a monic vanishing polynomial of minimal degree and study its properties over an infinite class of rings of a specific form.
In Section~\ref{van_of_prime}, we examine vanishing polynomials over rings of prime power (not necessarily integral domains or integer modulo rungs). 
Finally, in Section~\ref{counting_roots} we examine the properties of rings in which polynomials over the ring have a finite number of roots, as well as provide a methodology to limit the possible number of roots of polynomials over a product ring.

\section{Preliminaries} \label{vanishing_polynomials}

Throughout this paper, we assume rings to be commutative and with identity unless otherwise specified. Let $R$ be a ring and $R[x]$ be the ring of polynomials with coefficients in $R$. We study objects in $R[x]$ known as \emph{vanishing polynomials}. In this section, we define vanishing polynomials and note their relation to polynomial functions which was mentioned in the previous section, which is key to many of their applications to technological fields. 

\begin{definition}
A \emph{polynomial} $F(x)$ in a polynomial ring $R[x]$ is a formal sum $$a_{n}x^{n}+a_{n-1}x^{n-1}+\cdots+a_{1}x+a_{0}$$ for some nonnegative integer $n$, where each $a_{i} \in R$ and $x$ is an indeterminate.
\end{definition}

\begin{definition}
A \emph{vanishing polynomial} $F(x) \in R[x]$ is a polynomial such that $F(a) = 0$ for all $a \in R$. By definition, $0$ itself is a vanishing polynomial.
\end{definition}

\begin{example}
Consider the polynomial $F(x)=x^2+x$ over $\z_{2}$. Notice that $F(0)=0$ and $F(1)=0$. Therefore, $x^2+x$ is a vanishing polynomial over $\z_{2}$. Even infinite rings can have a finite nonzero vanishing polynomial: consider the ring $R = \Pi_{n=1}^{\infty}\z_{2}$. Notice that $x^2+x$ is vanishing in this ring as well. In fact, $ax^2+ax$ is a vanishing polynomial for any $a \in R$. 
\end{example}

The set of vanishing polynomials over a ring $R[x]$ is known to form an ideal. Thus, taking $I$ to be the ideal of vanishing polynomials, it makes sense to consider the quotient $R[x]/I$. In order to understand the significance of this operation we must first consider the distinction between a \emph{polynomial} and a \emph{polynomial function}. 

\begin{definition}
A \emph{polynomial function} $f : R \to R$ is a function on $R$ for which there exists a polynomial $F(x) \in R[x]$ such that $f(r)=F(r)$ for all $r \in R$.
\end{definition}

When distinguishing between polynomials and polynomials functions, we refer to the polynomials with uppercase letters, such as $F(x)$, and the corresponding polynomial functions with a lowercase letter, such as $f(x)$. Note that there is a natural injective mapping between polynomials $F(x)$ and polynomial functions $f(x)$ over any given ring, but each polynomial function $f(x)$ corresponds to an infinite number of polynomials, assuming the ring has a nonzero vanishing polynomial.
The following is a statement found in \cite{gilmer1999ideal} connecting polynomials and polynomial functions.

\begin{proposition}
\label{quotient}
The quotient $R[x] / I$, where $I$ is the ideal of vanishing polynomials over $R[x]$, is isomorphic to the ring of polynomial functions $f:R \to R$. 
\end{proposition}

An additional important result that we utilize is the following lemma, which allows us to decompose polynomials over product rings. A proof of the result is given in \cite{borodin2023ideal}.

\begin{lemma} \label{polynomial_product}
Given $R \cong R_{1} \times \cdots \times R_{k}$, we have
$$R[x] \cong R_{1}[x] \times \cdots \times R_{k}[x].$$
\end{lemma}

\section{Generalization of Monic Vanishing Polynomials} \label{general_of_monic}

In this section, we study the monic vanishing polynomial of minimal degree over a general ring $R$: define $s(R)$ to be the minimum positive integer $m$ such that there exists a polynomial $P(x)\in R[x]$ of degree less than $m$ for which $P(r)=r^m$ for all elements $r\in R$. In other words, $s(R)$ denotes the minimal $m$ for which the polynomial function $x\mapsto x^{m}$ corresponds to a polynomial in $R[x]$ of degree less than $m$. 

We can similarly define an extension of this notion to subrings of $R[x]$. Define $s(S;R)$, where $S$ is a subring of $R$, to equal the minimal degree $m$ such that a polynomial $P(x)\in S[x]$ of degree less than $m$ corresponds to the polynomial function $x\mapsto x^m$. In essence, the distinction between $s(S;R)$ and $s(R)$ is that $s(S;R)$ incorporates the additional restriction that the coefficients of $P(x)$ must be in the subring $S$, rather than the more relaxed condition of being in $R$.

These two characteristic numbers of a ring, $s(R)$ and $s(S;R)$, are introduced in \cite{specker2021polyfunctions}, where the case when $S=R',$, in which $R'$ denotes the subring generated by $1$, is specifically studied, and various number-theoretic and combinatorial bounds are given on the value of $s(R';R)$ if $s(R)$ is finite. In this section, we directly compute the value of $s(R)$ for an infinite class of rings that satisfy a specific form and use this to compute a bound on $s(R';R)$ for this class of rings, generalizing a result presented in \cite{specker2021polyfunctions}.

Let $U$ denote the set of all rings of the form $\mathbb{Z}_2[x]/(x^{a}+x^{a+1}),$ where $a\geq 3.$ 

\begin{lemma} \label{lem1}
For all rings $R$ in the set $U$, we have $s(R)\leq 4$.
\end{lemma}

\begin{proof}

Note that this proof can essentially be thought of as multiplying the modular equation presented in Lemma 8 of \cite{specker2021polyfunctions} through by $x^{a-3},$ but we present the full proof for completeness.

It suffices to show that for all $P\in \mathbb{Z}_2[x],$ it holds that

\begin{equation}\label{the_poly}
x^{a-2}P + (x^{a-3}+x^{a-2})P^2 + x^{a-3}P^4\equiv 0\pmod{x^a+x^{a+1}}.
\end{equation}

Consider when $P=x^k$ for a positive integer $k$. Equation~\ref{the_poly} evaluates to

\begin{align*}
&x^{a-2}(x^k)+(x^{a-3}+x^{a-2})(x^{2k})+x^{a-3}(x^{4k})
\\ &= x^{a+k-2} + x^{a+2k-3} + x^{a+2k-2} + x^{a+4k-3}.
\end{align*}
%$$x^{a-2}(x^k)+(x^{a-3}+x^{a-2})(x^{2k})+x^{a-3}(x^{4k}) = x^{a+k-2} + x^{a+2k-3} + x^{a+2k-2} + x^{a+4k-3}.$$

To verify that this expression vanishes $\pmod{x^{a}+x^{a+1}},$ we do casework based on the value of $k$.

The cases when $k=0$ and $k=1$ are easy to verify manually. When $k=0,$ the expression is equal to $x^{a-2} + x^{a-3}+x^{a-2} + x^{a-3}\equiv 0\pmod{x^a+x^{a+1}}$, and when $k=1,$ the expression is equal to $x^{a-1}+x^{a-1}+x^{a}+x^{a+1}\equiv 0\pmod{x^a+x^{a+1}}$.

When $k\geq 2,$ the expression can be greatly simplified by noting that $\pmod{x^a+x^{a+1}},$ we have that $x^{a}\equiv -x^{a+1}.$ Since all the coefficients of the terms are in $\mathbb{Z}_2[x]/(x^{a}+x^{a+1}),$ it follows that $-x^{a+1}$ is identically equal to $x^{a+1}$. Thus, we have that

$$x^a\equiv x^{a+1}\equiv x^{a+2}\equiv x^{a+3}\equiv \cdots.$$

If $k\geq 2,$ then $a+k-2, a+2k-3, a+2k-2,$ and $a+4k-3$ are all at least $a,$ so it follows that the expression is equal to $x^a+x^a+x^a+x^a\equiv 0\pmod{x^a+x^{a+1}}$.

Suppose Equation~\ref{the_poly} vanishes for $P_1$ and $P_2$. Then, due to the additive properties of $\mathbb{Z}_2[x],$ we have
\begin{align*}
&x^{a-2}(P_1+P_2)+(x^{a-3}+x^{a-2})(P_1+P_2)^2
\\ &+x^{a-3}(P_1+P_2)^4 \\
&=\sum_{i=1}^{2} x^{a-2}P_i+(x^{a-3}+x^{a-2})P_i^2+x^{a-3}P_i^4
\\ & \equiv 0\pmod{x^a+x^{a+1}}, 
\end{align*}
demonstrating that $P_1+P_2$ vanishes in Equation~\ref{the_poly} as well. Thus, any additive combination of monomials vanishes in Equation~\ref{the_poly}, and since all polynomials in $\mathbb{Z}_2[x]$ can be expressed as the sum of monomials, we have shown that Equation~\ref{the_poly} vanishes for all $P\in \mathbb{Z}_2[x],$ as desired.

\end{proof}

\begin{lemma} \label{lem2}
For all rings $R$ in the set $U$, we have $s(R)\geq 4$.
\end{lemma}

\begin{proof}
For the sake of contradiction, suppose that $s(R)\leq 3$. In other words, suppose that there exist coefficients $b_0, b_1, b_2, b_3\in R$ (not necessarily nonzero) for which 
\begin{equation}\label{another_poly}
b_0+b_1P+b_2P^2+b_3P^3 \equiv 0\pmod{x^a+x^{a+1}}
\end{equation}
for all $P\in \mathbb{Z}_2[x]$. Then, in particular, Equation~\ref{another_poly} must vanish when $P=0, 1, x, 1+x$.

When $P=0,$ Equation~\ref{another_poly} evaluates to $b_0,$ so $b_0=0$.

Substituting $P=1+x,$ we simplify the expression:

\begin{align}
&b_1(1+x)+b_2(1+x)^2+b_3(1+x)^3 \nonumber
\\ &= b_1(1+x) + b_2(1+x^2)+b_3(1+x+x^2+x^3) \nonumber
\\ &\equiv 0\pmod{x^a+x^{a+1}}. \label{equa3}
\end{align}

We can also substitute $P=1$ and $P=x$, respectively, to yield the following equations: 

\begin{align}
b_1+b_2+b_3&\equiv 0\pmod{x^a+x^{a+1}},\label{equa1}
\\ b_1x+b_2x^2+b_3x^3&\equiv 0\pmod{x^a+x^{a+1}}. \label{equa2}
\end{align}

We can then subtract Equation~\ref{equa1} and Equation~\ref{equa2}, respectively, from Equation~\ref{equa3}, yielding $b_3(x+x^2)\equiv 0\pmod{x^a+x^{a+1}}$.

As a result, $b_3$ is forced to be divisible by $x^{a-1}$, but since Equation~\ref{another_poly} must be monic according to the definition of $s(R),$ this implies that $b_3=0$. 

Equation~\ref{equa1} now simplifies to $b_1+b_2\equiv 0\pmod{x^a+x^{a+1}}$ and Equation~\ref{equa2} simplifies to $b_1x+b_2x^2\equiv 0\pmod{x^a+x^{a+1}}$. The former expression yields that $b_1=b_2,$ and substituting this into the latter yields $b_2(x+x^2)\equiv 0\pmod{x^a+x^{a+1}}$. Using analogous logic as we did to conclude that $b_3=0,$ we see that $b_2=b_1=0$. 

So we have that $b_3=b_2=b_1=b_0=0,$ a contradiction (Equation~\ref{another_poly} no longer has a well-defined degree), as desired.

\end{proof}

\begin{theorem}
For all rings $R$ in the set $U,$ we have $s(R) = 4$.
\end{theorem}

\begin{proof}
This immediately follows from Lemma~\ref{lem1} and Lemma~\ref{lem2}
\end{proof}

We now obtain a corresponding upper bound for $s(R';R)$ based on the value of $s(R)$. Let $\text{lcm}(n)$ denote $\text{lcm}[1,2,\ldots, n].$ The theorem below is from \cite{specker2021polyfunctions}.

\begin{theorem}
\cite[Theorem 7]{specker2021polyfunctions}
If $x=s(R)$ is finite, then $s(R';R)\leq \emph{lcm}(n)+n,$ where $n=x!^{(2x)^{x}x}$
\end{theorem}

We can apply the above theorem to the rings in $U$.

\begin{corollary} \label{bestlem}
We have $s(R';R)\leq \emph{lcm}(24^{2^{14}}) + 24^{2^{14}}$ for all rings $R$ in $U$.
\end{corollary}

Although this bound may seem large for smaller values of $a$, an interesting property about this bound is that it does not depend on the value of $a$. In other words, as the value of $a$ becomes increasingly large, the bound is not affected. Thus, even when the size of the ring approaches infinity, Corollary~\ref{bestlem} provides a finite bound for the value of $s(R';R)$.

\section{Vanishing Polynomials Over Rings of Prime Power Order} \label{van_of_prime}

In \cite{borodin2023ideal}, it is established that if a ring $R$ is a direct product of rings $R_{1}, R_{2}, \ldots, R_{n}$, then its vanishing polynomial ideal $I$ is a direct product of the vanishing polynomial ideals of $R_{1}, R_{2}, \ldots, R_{n}$. As a consequence, if we have a unique representation of all vanishing polynomials over $R_{1}, R_2, \ldots, R_n,$ we have a unique representation of all vanishing polynomials over $R$.

If we wish to find unique representations of all vanishing polynomials over all rings, it is important to find such representations for all finite commutative rings with identity. We also know that every finite commutative ring with identity can be expressed as the direct product of local commutative rings with identity of prime power order, as given in \cite{bini2002finite}. Therefore, if we find unique representations of all vanishing polynomials over all commutative rings of prime power orders with identity, we have such representations over all finite commutative rings with identity.

\subsection{Rings of order $p$}

Naturally, we begin by considering the simplest case of rings of prime order $p$, since proving or disproving a property for all rings of prime order will cover a large portion of the general case. Because $p$ is prime, we know that all commutative rings of prime order with identity are isomorphic to $\mathbb{Z}_{p}$, over which we already have a description of all the vanishing polynomials.

\subsection{Rings of order $p^2$}

We can now consider rings of order $p^2$, where $p$ is prime. This case is more nuanced than the last, but still within reach: every ring of order $p^2$ is isomorphic to either $\mathbb{Z}_{p} \times \mathbb{Z}_{p}$, $\mathbb{Z}_{p^2}$, $\text{GF}(p^2)$, or $\mathbb{Z}_{p}[x]/(x^2)$. We examine these subcases one by one. 

First, $\mathbb{Z}_{p} \times \mathbb{Z}_{p}$ is a direct product of two rings that we already have desired representations for, so we have such a representation for $\mathbb{Z}_{p} \times \mathbb{Z}_{p}$. 

Second, $\mathbb{Z}_{p^2}$ is a ring of integers modulo an integer, so we already have a unique representation for its vanishing polynomials. 

The third subcase is $\text{GF}(p^2)$ (the finite field of $p^2$ elements), whose elements we represent as $f_{1}, \ldots, f_{p^2}$. We now analyze this case.

Let $V(x)=(x-f_{1})(x-f_{2})\cdots(x-f_{p^2})$, where $\{f_{1}, f_{2}, \ldots, f_{p^2}\}$ are all the elements of $\text{GF}(p^2)$. Then, $V(x)$ is a vanishing polynomial, so all multiples of $V(x)$ are vanishing polynomials. Also, since there are no zero divisors in $\text{GF}(p^2)$, all vanishing polynomials are necessarily multiples of $(x-f_1)$, $(x-f_{2})$, \ldots, $(x-f_{p^2})$. Therefore, all vanishing polynomials are necessarily multiples of $V(x)$. Thus, there exists a one-to-one correspondence between vanishing polynomials and multiples of $V(x)$, so any vanishing polynomial $G(x)$ in $\text{GF}(p^2)$ can be uniquely represented as 

$$G(x)=F(x)V(x),$$

where $F(x)$ is a polynomial which is uniquely defined based on $G(x)$.

The fourth subcase $\mathbb{Z}_{p}[x]/(x^2)$ is the trickiest. To summarize the structure of the ring, it can be interpreted as the ring $\mathbb{Z}_{p}$ adjoined with a square nilpotent element.
Because $x$
appears as an element of $\mathbb{Z}_{p}[x]/(x^2)$, we will use the indeterminate $y$ when dealing with polynomials over $\mathbb{Z}_{p}[x]/(x^2)$. For example, $F(y)=yx$ is a polynomial $F(y) \in (\mathbb{Z}_{p}[x]/(x^2))[y]$, and it maps $1$ to $x$ and $x$ to $0$.

\begin{definition}
Let polynomial $P(y)=a_{n}y^{n}+ \dotsb + a_{1}y + a_{0}$, where $n \geq 2$. Then, the \textit{formal derivative} is defined as $P'(y)=na_{n}y^{n-1}+ \dotsb + 2a_{2}y + a_{1}$.
\end{definition}

\begin{proposition}
Let $J$ and $K$ be vanishing polynomials over $\mathbb{Z}_{p}$. Then, the polynomial $$G(y)=J(y) + x \cdot K(y)$$
vanishes in $\mathbb{Z}_{p}[x]/(x^2)$ if and only if $J'$ vanishes over $\mathbb{Z}_{p}$.
\end{proposition}

\begin{proof}
Let $G(y)$ be a vanishing polynomial in $\mathbb{Z}_{p}[x]/(x^2)$. Then, $$G(y)=\sum_{i=0}^{n} C_{i}y^{i}.$$
Substituting $y = Ax + B$ and $C_{i}=D_{i}x+E_{i}$ we have
\begin{align*}
    &G(y) \\
    &=\sum_{i=0}^{n} (D_{i}x+E_{i})(Ax + B)^{i} \\
    &= \sum_{i=1}^{n} (D_{i}x+E_{i})(Ax + B)^{i} + D_{0}x+E_{0} \\
    &= \sum_{i=1}^{n} (D_{i}x+E_{i})(iAB^{i-1}x + B^{i}) + D_{0}x+E_{0} \\
    &= \sum_{i=1}^{n}E_{i}B^{i}+E_{0}+x\sum_{i=1}^{n}(D_{i}B^{i}) +D_{0}x + xA\sum_{i=1}^{n}iB^{i-1}E_{i} \\
    &= \sum_{i=0}^{n}E_{i}B^{i} + x\sum_{i=0}^{n}(D_{i}B^{i}) +
    xA\sum_{i=0}^{n}(i+1)E_{i+1}B^{i} \\
    &= J(B) + x \cdot K(B) + x \cdot A \cdot J'(B)
\end{align*}
where $J(B)$ and $K(B)$ are polynomials in $\mathbb{Z}_{p}$. We immediately see that $J(B)$ is vanishing, and substituting $A=0$, we obtain that $K(B)$ is vanishing. Substituting $A=1$, we obtain that $J'(B)$ is also vanishing.
Now, substituting $B=y$ and $A=0$, we obtain $G(y)=J(y) + x \cdot K(y)$.

The converse is clear: let $A,B$ be elements of $\mathbb{Z}_{p}$ and $y=Ax+B$, and suppose $J$, $K$, $J'$ are vanishing over $\mathbb{Z}_{p}$. Then, $G(y)= J(B) + x \cdot K(B) + x \cdot A \cdot J'(B) = 0$. Therefore, since all elements of $\mathbb{Z}_{p}[x]/(x^2)$ can be represented as $Ax+B$, $G$ is vanishing on $\mathbb{Z}_{p}[x]/(x^2)$.
\end{proof}

\begin{example}
To illustrate this theorem, we will now consider some examples over $\mathbb{Z}_{2}[x]/(x^2)$.
\begin{enumerate}
    \item Suppose $J(y)=0$ and $K(y)=y(y+1)$. Then, both $J$ and $K$ vanish over $\mathbb{Z}_{2}$. In addition, $J'(y)=0$ vanishes over $\mathbb{Z}_{2}$. Now, notice that $G(y)=J(y)+x \cdot K(y)=xy(y+1)$ vanishes over $\mathbb{Z}_{2}[x]/(x^2)$.
    \item Suppose $J(y)=y^4+y^2$ and $K(y)=0$. It follows that both $J$ and $K$ vanish over $\mathbb{Z}_{2}$. Furthermore, $J'(y)=4y^3+2y=0$ vanishes over $\mathbb{Z}_{2}$. Now, notice that $G(y)=J(y)+x \cdot K(y)=y^4+y^2$ vanishes over $\mathbb{Z}_{2}[x]/(x^2)$.
    \item Suppose $J(y)=y(y+1)$ and $K(y)=0$. Then, both $J$ and $K$ vanish over $\mathbb{Z}_{2}$. However, $J'(y)=2y+1=1$ does not vanish over $\mathbb{Z}_{2}$. Now, notice that $G(y)=J(y)+x \cdot K(y)=y(y+1)$ does not vanish over $\mathbb{Z}_{2}[x]/(x^2)$: for a counterexample, see $K(x)=x \neq 0$.
\end{enumerate}
\end{example}

% Currently we have not found unique representations of all vanishing polynomials for all commutative rings of order $p^2$ with identity. However, if we are able to, we will then examine the $p^3$ case, then the $p^4$ case, and so on. In particular, all commutative rings of order $p^3$ with identity are isomorphic to one of a handful of specific rings, so we can divide the $p^3$ case up into subcases and solve for each of the subcases as we have done above.

\section{Counting Roots of Polynomials} \label{counting_roots}

Before considering the problem of counting and bounding the number of roots of polynomials over rings, we start by considering a weaker notion, namely, considering whether the number of roots of these polynomials is finite.

More specifically, we aim to study what characteristics we can discern about a commutative nonzero ring (not necessarily with identity) with the property that every polynomial has a finite number of roots. In the case where the ring is finite, it easily follows that every polynomial has a finite number of roots as there are only a finite number of elements in the ring. In the case where the ring is infinite, the following theorem provides a characterization of such rings which satisfy the desired property. A similar result was given in \cite{drike_moos_dan_2016}, but we extend their proof to include rings not necessarily with identity. 

\begin{theorem} \label{thethm}
The only infinite rings $R$ which satisfy the property that all polynomials in $R[x]$ have finitely many roots are rings with no nonzero zero divisors.
\end{theorem}

\begin{proof}
Suppose that the infinite ring $R$ satisfies the given property. We wish to show that this implies that $R$ has no nonzero zero divisors.

Suppose for contradiction that $R$ does have a zero divisor, call it $a$. Then the roots of the polynomial $ax$ over $R$ are by definition the annihilator of $a$ over $R$ (denoted $\text{Ann}(a)$), which is well known to be an ideal of $R$. By our assumption that all polynomials with coefficients in $R$ have a finite number of roots, we have that the size of $\text{Ann}(a)$ is finite.

Let a nonzero element in $\text{Ann}(a)$ be denoted by $y$. It follows that $yR$ is finite because $yR\subseteq \text{Ann}(a),$ and we have already shown that $\text{Ann}(a)$ is finite. Consider the R-module homomorphism $\phi:R\to yR$ given by $\phi(r)=yr$. Because $yR$ is finite and $R$ is infinite, it follows that there is an element $yr$ in $yR$ that has infinitely many elements of $R$ mapping to it. Denote the set of elements in $R$  mapping to $yr$ by $M$. If $M$ is countably infinite, then denote the elements in $M$ by $\{r_1, r_2, r_3, \ldots\}$. Otherwise, we can choose a countably infinite subset of $M$, and denote the elements in this subset by $\{r_1, r_2, r_3, \ldots\}$; this follows directly from the axiom of choice.

Thus, for all $n\geq 2,$ we have that $yr_1=yr_n,$ and so $y(r_1-r_n)=0$. But because all the elements in $\{r_1, r_2, r_3, \ldots\}$ are distinct, it follows that all the elements in $\{r_1-r_n\}_{n=2}^{\infty}$ are also distinct, so there are an infinite number of roots to the equation $yx=0$. However, since $yx\in R[x],$ it is supposed to have a finite number of roots, a contradiction.

Thus, infinite rings which have zero divisors do not satisfy the condition that all polynomials in the ring have a finite number of roots.

We now show that if a ring has no zero divisors, then all polynomials over the ring have a finite number of roots. Suppose that $R$ has no zero divisors, and let $F(x)$ denote a polynomial in $R[x]$. Note that $K=\{\frac{a}{b} \mid a,b\in R, b\neq 0\}$ is a field. $F(x)$ is an element of $R[x],$ which is contained in $K[x]$. Since every polynomial over a field has a finite number of roots, it follows that $F(x)$ has a finite number of roots over $R[x],$ as desired.
\end{proof}

We now present a result that allows us to limit the number of roots a polynomial could have over a commutative ring that can be represented as a ring product.

Consider a finite commutative ring $R$ with identity which can be written as $R \cong R_1 \times R_2 \times \cdots \times R_k$. Thus, each element $x \in R$ can be represented as some tuple $(x_1, x_2, \ldots, x_k)$ where each $x_i$ is an element of $R_i$. Let us denote $n = |R|$ and $n_1 = |R_1|, n_2 = |R_2|, \ldots, n_k = |R_k|$. Take any polynomial $F: R \to R$. Equivalently, we can view $F$ as a $k$-tuple of polynomials $F_1, F_2, \ldots, F_k$ over $R_1, R_2, \ldots R_k$, respectively, as described by Lemma~\ref{polynomial_product}. Note that if an element $x$ is a root of $F$, that is, if $F(x) = 0$, it must be a root of $F_1, F_2,\ldots, F_k$ in each of $R_1, R_2, \ldots R_k$. However, rather than counting roots, we count elements that are not roots. Let us define $0 \leq a_i$ to be the number of elements that are not roots of $F_i$ in $R_i$ for $1\leq i\leq k$. To find the number of non-roots of $F$ over $R$, we must find the number of tuples $(F_1(x_1), F_2(x_2), \ldots, F_k(x_k))$ where not all elements are 0. 

We can use the principle of inclusion-exclusion to calculate this number. Each set of $a_i$ non-roots over $R_i$ contributes $(n/n_i)\cdot a_i$ tuples that are non-zero. However, all tuples where two entries are non-zero are double counted so we must subtract $(n/(n_i\cdot n_j))\cdot (a_i\cdot a_j)$ for every pair of $i, j$. Now all tuples that contain three non-zero entries have been added three times and subtracted three times so they must be added back. Continuing in this manner we find that the number of non-roots of $F$ over $R$ is 
$$
\sum_{i = 1}^{k} \sum_{j_1<j_2<\ldots< j_i\leq k} \frac{(-1)^{k+1}n\cdot (a_{j_1}\cdot a_{j_2} \cdot \ldots \cdot a_{j_i})}{n_{j_1}\cdot n_{j_2} \cdot \ldots \cdot n_{j_i}}.
$$
Thus, we can conclude that if a number cannot be represented in this way for any set of $a_1, a_2, \ldots, a_k$, a polynomial cannot have that number of non-roots. The converse is not necessarily true, since over $R_i$, not every number of roots of a polynomial is a possible $a_i$. Note that since we know $|R|$, if we know a polynomial cannot have $q$ non-roots over $R$, it directly follows that a polynomial cannot have $|R| - q$ roots. 

In the case of the squarefree integers, however, any function is possible over a field such as $\Z{p}$ for prime $p$, so all $a_i$'s are possible and the converse of the above statement holds. 

As an example, consider $\Z{6} = \Z{3} \times \Z{2}$. In this case every polynomial must have $2a_1 + 3a_2 - a_1a_2$ non-roots for some integers $0 \leq a_1 \leq 3$ and $0 \leq a_2 \leq 2$. Thus, one may easily verify that a polynomial over $\Z{6}$ cannot have exactly 1 non-root but can have 0, 2, 3, 4, 5, or 6 non-roots. By complimentary counting, we can conclude that a polynomial over $\Z{6}$ can have 0, 1, 2, 3, 4, or 6 roots but cannot have 5 roots.

\section{Acknowledgements}

We would like to thank our research advisor, Prof. James Barker Coykendall, for mentoring us throughout this research process. We would also like to thank Dr. Felix Gotti and Dr. Tanya Khovanova for reviewing the paper and providing many helpful comments for revision. Finally, we would like to express our gratitude to Dr. Pavel Etingof, Dr. Slava Gerovitch, and the MIT PRIMES program for making this project possible.

\end{document}